\documentclass{amsart}
\usepackage{amssymb,hyperref} 
\newtheorem{thm}{Theorem}[section]
\newtheorem{cor}[thm]{Corollary}
\newtheorem{lem}[thm]{Lemma}

\theoremstyle{definition}

\newtheorem{con}[thm]{Conjecture}
\numberwithin{equation}{section}

\newcommand{\pr}[1]{\left(#1\right)}

\begin{document}

\title[Compact groups with many elements of bounded order]{Compact groups with many elements of bounded order}%
\author[M. Soleimani Malekan]{Meisam Soleimani Malekan}%
 \address{Department of Mathematics, University of Isfahan, Isfahan, Iran} 
\email{m.soleimani@sci.ui.ac.ir}
\author[A. Abdollahi]{Alireza Abdollahi}%
\address{Department of Mathematics, University of Isfahan, Isfahan 81746-73441, Iran; Institute for Research in Fundamental Sciences, School of Mathematics, Tehran, Iran.} 
\email{a.abdollahi@math.ui.ac.ir}%
\author[M. Ebrahimi]{Mahdi Ebrahimi}
\address{Institute for Research in Fundamental Sciences, School of Mathematics, Tehran, Iran.} 
\email{m.ebrahimi.math@gmail.com}

\begin{abstract}
L\'evai  and  Pyber proposed the following as a conjecture:\\  
 Let $G$ be a profinite group such that the set of solutions of the equation $x^n=1$
has positive Haar measure. Then $G$ has an open subgroup $H$ and an element $t$ such that all
elements of the coset $tH$ have order dividing $n$ (see Problem 14.53 of [The Kourovka Notebook, No. 19, 2019]). \\
The validity of the conjecture has been proved in [Arch. Math. (Basel) 75 (2000) 1-7] for $n=2$. Here we study the conjecture for compact groups $G$ which are not necessarily profinite and $n=3$; we show that in the latter case the group $G$ contains an open normal $2$-Engel subgroup. 
\end{abstract}
\maketitle
\section{Introduction and Results}

 Every compact Hausdorff topological group $G$ admits a unique normalized Haar measure  $\mathfrak{m}_G$.  If $X$ is a measurable subset of $G$ such that $\mathfrak{m}_G(X)>0$, then it is not true in general that $X$  contains a non-empty open subset; the latter is not even true for profinite groups i.e.,  compact totally disconnected Hausdorff topological groups, see e.g. \cite{LP}. However, as far as we know, there is no counterexample for the latter question for certain subsets of positive Haar measure defined by words in profinite groups. 
In this direction, the following conjecture has been proposed by L\'evai and Pyber in \cite{LP}.

\begin{con}[Conjecture 3 of \cite{LP}, Problem 14.53 of \cite{Ko}] \label{con} Let $G$ be a profinite group such that the set of solutions of the equation $x^n=1$
has positive Haar measure. Then $G$ has an open subgroup $H$ and an element $t$ such that all
elements of the coset $tH$ have order dividing $n$.
\end{con}

The validity of Conjecture \ref{con} was confirmed for $n=2$ in \cite{LP}. We generalize the latter for compact (Hausdorff) groups for the set of elements inverting by an automorphism (not necessarily continuous) of a compact group, see Theorem \ref{invthm} below. Here we also study Conjecture \ref{con} for $n=3$. However we were not able to settle Conjecture \ref{con} for $n=3$,  we prove that  compact groups (not necessarily profinite) in which the set of elements of order dividing $3$ is of positive Haar measure, contain open normal $2$-Engel subgroups, see Corollary \ref{Engel} below; actually we prove a more general result that the latter is valid for compact groups in which the set of elements that an automorphism of order $3$ acts on as an splitting automorphism, has positive Haar measure, see   Theorem \ref{spilitEngel} below. 

 Our main tool (Theorem \ref{key} below) to deal with compact groups is proved in Section \ref{sec1}. Theorem \ref{key} is a general result of independent interest on measurable subsets of compact groups with positive Haar measure. Theorem \ref{key} in particular shows that   the latter subsets are ``relatively  $k$-large sets" in compact groups (see \cite{JW, JW2}; for definition of ``$k$-large sets" and for some results on them see \cite{MK}).   

\section{Subsets of compact groups with positive Haar measure are Large}\label{sec1}

Throughout all topological groups are Hausdorff. 
Let $G$ be a compact group with the unique normalized Haar measure ${\mathfrak m}_G$; usually the index $G$ is dropped if the group is known from the context. 
It follows from Corollary 20.17 of \cite{hr} that for measurable subsets $A$ and $B$ of $G$ with ${\mathfrak m}(A)>0$ and ${\mathfrak m}(B)>0$, the map from $G$ to ${\mathbb R}^{\geq 0}$ defined by $x\mapsto {\mathfrak m}(A \cap xB)$ is non-zero and continuous. Following a question proposed by A. Abdollahi in \cite{MOF}, we learned that the following generalization of the latter result holds.  
\begin{thm}\label{key}
Let $G$ be a compact group. Suppose that $A_1, \dots, A_n$ are measurable subsets of $G$ with positive Haar measure. Then the map 
\[ \Lambda : \underset{n{\rm -times}}{G\times \cdots\times G}\rightarrow\mathbb R^{\geq0},\quad (x_1, \dotsc, x_k)\mapsto {\mathfrak m}\big(x_1A_1 \cap \cdots \cap x_nA_n\big)\] 
is non-zero and continuous. In particular, if $A$ is a measurable subset with positive Haar measure, then for any positive integer $k$ there exists an open subset $U$ of $G$ containing $1$ such that $\mathfrak{m}_G(A \cap u_1 A \cap \cdots \cap u_k A)>0$ for all $u_1,\dots,u_k \in U$.    
\end{thm}
Theorem \ref{key} can be obtained from a more general result as follows; this result, in a sense, is a generalization of the convolution product: \\
\begin{thm}\label{MSM}
Let $G$ be a compact group and let $\xi_1\dots,\xi_n$ be elements in the unit ball of $L^\infty(G)$. Then the map 
\[ \Psi:\underset{n{\rm -times}}{G\times \cdots\times G} \rightarrow\mathbb{C}, \quad (x_1,\dotsc,x_n)\mapsto\int_G\pr{\prod_{k=1}^nL_{x_k}\xi_k(g)}\,\text d{\mathfrak m}(g)\] 
is continuous.
\end{thm}
{\noindent\bf Proof of Theorem \ref{MSM}.}
The Cauchy-Schwarz inequality in the Hilbert space $L^2(G)$ can be stated in the following integral form: 
\begin{align}\label{C-S ineq}
\left|\int_G \xi(g)\overline{\eta(g)}\,\text d{\mathfrak m}(g)\right|\leq\|\xi\|_2\|\eta\|_2\quad(\xi,\eta\in L^2(G)).
\end{align} 
The left translate of a function $f$ on $G$ by an $x\in G$, $L_xf$, is defined by 
\[ L_xf(g)=f(x^{-1}g)\quad(g\in G).\]
The following is a special case of \cite[(2.41) Proposition.]{folland}:
\begin{lem}\label{uniformcontiniuty} 
Let $\xi$ be in $L^2(G)$. Then the map 
\[ G\rightarrow L^2(G), \quad g\mapsto L_x\xi\] 
is continuous.
\end{lem}

{\bf Proof of Theorem \ref{MSM}.}
For $x_k,y_k$, $k=1,\dots,n$, in $G$, and $1\leq s\leq n$, put 
\[F_s(g):=\pr{\prod_{k<s}L_{y_k}\xi_k(g)}\pr{\prod_{k>s}L_{x_k}\xi_k(g)}\quad(g\in G).\]
If the index of the latter product is empty, we put 1 instead of the product. By the latter convention,  
\[\Psi(x_1,\dots, x_n)-\Psi(y_1,\dots,y_n)=\sum_{s=1}^n\int_G\pr{L_{x_s}\xi_s(g)-L_{y_s}\xi_s(g)}F_s(g)\,\text d{\mathfrak m}(g).\]
It follows from the Cauchy-Schwarz inequality, (\ref{C-S ineq}), and the fact that $\|F_s\|_2\leq1$,  
\[ |\Psi(x_1,\dotsc,x_n)-\Psi(y_1,\dotsc,y_n)|\leq\sum_{s=1}^n\left\|L_{x_s}\xi_s-L_{y_s}\xi_s\right\|_2\] 
whence by Lemma \ref{uniformcontiniuty}  the continuity of $\Psi$ follows. $\hfill \Box$ \\

{\bf Proof of Theorem \ref{key}.}
Apply Theorem \ref{MSM} for $\xi_k:=\chi_{A_k}$, $k=1,\dotsc, n$, and note that by Fubini's theorem, we have 
\[\int_{\prod_{i=1}^nG}{\mathfrak m}\big(  x_1A_1 \cap \cdots \cap x_nA_n  \big)\text d{\mathfrak m}_{\prod_{i=1}^nG}(x_1,\dotsc,x_n)={\mathfrak m}(A_1)\cdots {\mathfrak m}(A_n)>0.  \]

Since $\epsilon:=\mathfrak{m}_G(A)>0$, it follows from the continuity  of the map $\Lambda$ corresponding to $A_i:=A$ ($i=1,\dots, k+1$)  there exists an open subset $\mathfrak{U}$ of $\underset{(k+1){\rm -times}}{G\times \cdots\times G}$ containing $(1,1,\dots,1)$ such that 
$\mathfrak{m}_G(A \cap u_1A \cap \cdots \cap u_k A)\geq \epsilon>0$ for all $(1,u_1,\dots,u_k)\in \mathfrak{U}$. Now let $U:=U_0 \cap U_1 \cap \cdots \cap U_k$, where $U_0\times U_1 \times \cdots \times U_k \subseteq \mathfrak{U}$ and $U_i$ ($i=0,1,\dots, k$) are open subsets of $G$ containing $1$, has the required property.  $\hfill \Box$ \\

Following \cite{JW} a subset $X$ of a group $G$ is called large if $\bigcap_{a\in F} a X$ is not empty for any finite non-empty subset $F\subseteq G$. We say that a subset $X$ of a group $G$ is called relatively $k$-large with respect to a subset $M$ of $G$ for some $k\in\mathbb{N}$ if  $\bigcap_{a\in F} a X$ is not empty for any subset $F\subseteq M$ with $|F|=k$. Thus, by Theorem \ref{key}, for every $k\in\mathbb{N}$, each measurable subset of a compact group with positive Haar measure is relatively $k$-large with respect to an open subset $U_k$ containing $1$. 

\section{Compact groups with an automorphism inverting many elements}

In this section we generalize \cite[Corollary 6]{LP}.

\begin{thm} \label{invthm}
Let $G$ be a compact group having an automorphism (not necessarily continuous) $\alpha$ such that the set $X=\{x\in G \;|\; x^\alpha=x^{-1}\}$ is measurable and of positive Haar measure. Then $G$ contains an open normal abelian subgroup. In particular, if $G$ is profinite, there exists an open abelian subgroup $A$ of $G$ such that $tA\subseteq X$ for some $t\in X$.   
\end{thm}
\begin{proof}
It follows from Theorem \ref{key}, there exists an open subset $U$ of $G$ containing $1$ such that $X \cap u_1 X \cap u_2X \cap u_3 X$ is of positive Haar measure for all 
$u_1,u_2,u_3\in U$. Since $U$ is open and $1\in U$, it follows from \cite[Theorem 4.5]{hr} that there exists an open subset $V\subseteq U$ such that $1\in V$, $V=V^{-1}=\{v^{-1} \;|\; v\in V\}$ and $V^2:=\{v_1v_2 \;|\; v_1,v_2\in V\}\subseteq U$ (see Thoerem 4.5 of \cite{hr}). Now suppose that $a,b$ are arbitrary elements of $V$. Thus  
\begin{equation*}
X \cap b^{-1}X \cap a^{-1}X \cap b^{-1}a^{-1}X
\end{equation*}
has positive Haar measure and in particular, it is non-empty. It follows that there exists $x\in X$ such that 
$bx,ax,abx$ are all in $X$. Therefore 
$x^\alpha=x^{-1}$ and so  $b^\alpha=b^x=b^{-1}$, $a^{\alpha}=a^x=a^{-1}$ and $(ab)^\alpha=(ab)^x=(ab)^{-1}$.
It follows that $[a,b]=1$ and so the subgroup $H$ generated by $V$ is abelian. 
Since $H=\bigcup_{i=1}^n V^n$, $H$ is open. Now take the core $K$ of $H$ in $G$, which is open normal and abelian. This completes the proof of the first part. 

For the second part, since $|G:K|$ is finite, $tK \cap X$ has positive Haar measure for some $t\in X$. Thus $A:=\{a \in K \;|\; ta \in X \}$ has positive Haar measure and since $K$ is abelian, $A$ is a subgroup of $G$. Since $X$ is closed, $A$ is closed and it is open since $A$ has positive Haar measure. This completes the proof.  
\end{proof}

\section{Compact groups with many elements of order $3$}

In this section we study Conjecture \ref{con} for $n=3$. 

The following lemma is used to prove for a relatively $8$-large set with respect to an open subgroup $U$ containing $1$, $U$ is $2$-Engel. 

\begin{lem}\label{2-Engel}
Let $a$, $b$ and $x$ be elements of a group such that 
\begin{equation*}
x^3=(bx)^3=(ax)^3=(a^{-1}x)^3=(ab^{-1}x)^3=(ba^{-1}x)^3=(abx)^3=(b^{-1}a^{-1}x)^3=1. 
\end{equation*}
Then $[a,b,b]=1$.
\end{lem}
\begin{proof}
See the proof of Proposition 4.3 of \cite{JW2} and proof of Proposition 1 of \cite{MK}. 
\end{proof}
The following two next lemmas give a sufficient condition on a symmetric subset of a group to generate a $2$-Engel subgroup.
\begin{lem}\label{nil}
There exists a positive integer $k$ such that every group  generated by a symmetric subset $X=X^{-1}$ containing $1$ satisfying the property  $[x,y,y]=1$ for all $x,y\in X^k:=\{x_1\cdots x_k \;|\; x_1,\dots, x_k \in X\}$ is nilpotent of class at most $3$.
\end{lem}
\begin{proof}
It is enough to show that every $4$-element subset of $X$ generates a nilpotent group of class at most $3$. Since $2$-Engel groups are nilpotent of class at most $3$ \cite[Corollary 3, page 45]{R2}, $\frac{F_4}{\langle [a,b,b] \;|\; a,b\in F_4 \rangle}$ is nilpotent of class at most $3$, where $F_4$ is the free group of rank $4$ on the free generators $x_1,x_2,x_3,x_4$. Now it follows from \cite[Lemma 1.43, page 32 and its Corollary, page 33]{R1} that  $\langle [a,b,b] \;|\; a,b\in F_4 \rangle$ is the normal closure of a finite number of its elements. Then, there exist elements $a_1,\dots, a_s, b_1,\dots,b_s \in F_4$ such that 
\begin{equation*} \langle [a,b,b] \;|\; a,b\in F_4 \rangle= \langle  [a_1,b_1,b_1],\dots, [a_s,b_s,b_s]\rangle^{F_4}.
\end{equation*}
Let $k$ be a positive integer such that  
$$\{a_1,\dots,a_s,b_1,\dots,b_s\}\subseteq \{1, x_1,x_2,x_3,x_4,x_1^{-1},x_2^{-1},x_3^{-1},x_4^{-1}\}^k.$$ 
Therefore, any group  satisfying the laws $[a_i,b_i,b_i]=1$ for all $i=1,\dots,s$ is nilpotent of class at most $3$ and since $a_i$s and $b_i$s are the product 
of at most $k$ free generators $x_j$s and their inverses, the proof is now complete. 
\end{proof}
\begin{lem}\label{eng}
There exists a positive integer $\ell$ such that every group  generated by a symmetric subset $X=X^{-1}$ containing $1$ satisfying the property  $[x,y,y]=1$ for all $x,y\in X^\ell$ is a $2$-Engel group.
\end{lem}
\begin{proof}
Let $\ell:=\max\{k,2\}$, where $k$ is a positive integer mentioned in the statement of Lemma \ref{nil}. By Lemma \ref{nil}, $G:=\langle X \rangle$ is nilpotent of class at most $3$. Suppose that $g$ and $h$ are arbitrary elements of  $G$. Then $g=x_1\cdots x_t$ and $h=y_1\cdots y_t$ for some $x_i,y_j\in X$. Now we may write
$[g,h,h]$ as follows 
$$\prod_{i=1}^t \big(\prod_{j<k,i,j=1}^t ([x_i,y_j,y_k][x_i,y_k,y_j]) \big).$$ 
The latter follows from the facts that $G$ is nilpotent of class at most $3$ and $[x,y,y]=1$ for all $x,y \in X$. Now since $\ell\geq 2$, $[x,yz,yz]=1$ for all $x,y,z \in X$. It follows that $[x,y,z][x,z,y]=1$. Therefore $[g,h,h]=1$ and so $G$ is $2$-Engel.
\end{proof}
\begin{thm}\label{spilitEngel}
Let $G$ be a compact group and $\alpha$ be a  not necessarily continuous automorphism of $G$  such that $\alpha^3=1$ and the set $X=\{x\in G \;|\; x^{\alpha^2} x^{\alpha} x=1\}$ is measurable. If $X$ has positive Haar measure, then $G$ contains an open normal  $2$-Engel subgroup. 
\end{thm}
\begin{proof}
By Lemma \ref{key} there exists an open subset $U$ of $G$ such that $X \cap u_1 X \cap \cdots \cap u_7 X$ is of positive Haar measure for all $u_1,\dots,u_7\in U$. Since $U$ is open and $1\in U$, it follows from \cite[Theorem 4.5]{hr} there exists an open subset $V\subseteq U$ such that $1\in V$, $V=V^{-1}$ and $V^{2\ell}\subseteq U$ . Now suppose that $a_1,\dots, a_\ell$ and $b_1,\dots b_{\ell}$ are arbitrary elements of $V$. Let $a=a_1\cdots a_\ell$ and $b=b_1\cdots b_\ell$. Thus  
\begin{equation*}
X \cap b^{-1}X \cap aX \cap a^{-1}X \cap ab^{-1}X \cap ba^{-1}X \cap abX \cap b^{-1}a^{-1}X
\end{equation*}
has positive Haar measure and in particular, it is non-empty. It follows that there exists $x\in X$ such that 
$bx,ax,a^{-1}x,ab^{-1}x,ba^{-1}x,abx,b^{-1}a^{-1}x$ are all in $X$. 
Working in the semidirect product $G \rtimes \langle \alpha \rangle$, since $\alpha^3=1$, it follows that  $a\in X$ if and only if $(a\alpha)^3=1$. 
Therefore $(x\alpha)^3=(bx\alpha)^3=1$ and $$(ax\alpha)^3=(a^{-1}x\alpha)^3=(ab^{-1}x\alpha)^3=(ba^{-1}x\alpha)^3=(abx\alpha)^3=(b^{-1}a^{-1}x\alpha)^3=1.$$
 Now Lemma \ref{2-Engel} implies that $[a,b,b]=1$ and  it follows from Lemma \ref{eng} that $H:=\langle V \rangle$ is $2$-Engel. Since $V$ is open, $H$ is a open. Now take the core of $H$ in $G$, which is normal open and $2$-Engel. This completes the proof.
\end{proof}
\begin{cor}\label{Engel}
Let $G$ be a compact group such that the set $X=\{x\in G \;|\; x^3=1\}$ has positive Haar measure. Then $G$ contains an open normal $2$-Engel subgroup. 
\end{cor}
\begin{proof}
Take $\alpha$ as the identity automorphism in Theorem \ref{spilitEngel}.
\end{proof}

\section{Acknowledgements}

The first author is grateful to National Elite Foundation of Iran for financial support. The second author  was supported in part
by a grant from School of Mathematics, Institute for Research in Fundamental Sciences (IPM). 

\end{document}